\documentclass[technote,twoside,web]{ieeecolor}
\usepackage{generic}
\usepackage{cite}
\usepackage{amsmath,amssymb,amsfonts}
\usepackage{algorithmic}
\usepackage{graphicx}
\usepackage{algorithm,algorithmic}
\usepackage{hyperref}
\hypersetup{hidelinks=true}
\usepackage{textcomp}

\let\labelindent\relax
\usepackage{enumitem}

\usepackage{ifthen}
\usepackage{url}

\usepackage{pgfplots}
\usepackage{pgf}
\usepackage{tikz} % package for graphics with tikz
\usetikzlibrary{decorations.pathmorphing} % package fuer Feder!
\usepackage{tikz-3dplot}
\usepgfplotslibrary{fillbetween}
\usetikzlibrary{arrows}
\usetikzlibrary{graphs,decorations.pathmorphing,decorations.markings}
\usetikzlibrary{calc,math}
\usetikzlibrary{patterns}
\usetikzlibrary{shapes}
\usetikzlibrary{tikzmark}

\usetikzlibrary {positioning}
\usetikzlibrary{shadows}
\usetikzlibrary{patterns.meta}
\usetikzlibrary{plotmarks}

\usepackage{subcaption}  %probably includes  \usepackage{subfigure}

\usepackage[absolute,overlay]{textpos}
\usepackage{scalefnt}   % Used in some tikz-figures

\tikzdeclarepattern{
  name=hatch,
  parameters={\hatchsize,\hatchangle,\hatchlinewidth},
  bounding box={(-.1pt,-.1pt) and (\hatchsize+.1pt,\hatchsize+.1pt)},
  tile size={(\hatchsize,\hatchsize)},
  tile transformation={rotate=\hatchangle},
  defaults={
    hatch size/.store in=\hatchsize,hatch size=5pt,
    hatch angle/.store in=\hatchangle,hatch angle=0,
    hatch linewidth/.store in=\hatchlinewidth,hatch linewidth=.4pt,
  },
  code={
      \draw[line width=\hatchlinewidth] (0,0) -- (\hatchsize,\hatchsize);
  }
}

\pgfmathsetmacro\weight{1/2}
\pgfmathsetmacro\third{1/3}
\pgfmathsetmacro\twothirds{2/3}

\tikzset{degil/.style={
            decoration={markings,
            mark= at position 0.5 with {
                  \node[transform shape] (tempnode) {$/$};
                  }
              },
              postaction={decorate}
}
}

\usetikzlibrary{shadows}
\tikzset{
diagonal fill/.style 2 args={fill=#2, path picture={
\fill[#1, sharp corners] (path picture bounding box.south west) -|
                         (path picture bounding box.north east) -- cycle;}},
reversed diagonal fill/.style 2 args={fill=#2, path picture={
\fill[#1, sharp corners] (path picture bounding box.north west) |- 
                         (path picture bounding box.south east) -- cycle;}}
}

\usepgfplotslibrary{fillbetween}
\usetikzlibrary{intersections,through}

\makeatletter % from https://tex.stackexchange.com/a/127045/121799
\tikzset{use path/.code=\tikz@addmode{\pgfsyssoftpath@setcurrentpath#1}}
\makeatother

\definecolor{manipulator-color}{RGB}{88,44,44}
\definecolor{manipulator-contour}{rgb}{0.0, 0.18, 0.39}  %cool black
\tikzset{>=latex} % for LaTeX arrow head

\newtheorem{theorem}{Theorem}[section]

\newtheorem{proposition}[theorem]{Proposition}

\newtheorem{ass}[theorem]{Assumption}

\newtheorem{nnremark}[theorem]{\bf Remark}
\newtheorem{nndefinition}[theorem]{\bf Definition}
\newenvironment{remark}{\begin{nnremark} \rm }{\end{nnremark}}
\newenvironment{definition}{\begin{nndefinition} \rm }{\end{nndefinition}}

\makeatletter
\DeclareOldFontCommand{\rm}{\normalfont\rmfamily}{\mathrm}
\DeclareOldFontCommand{\sf}{\normalfont\sffamily}{\mathsf}
\DeclareOldFontCommand{\tt}{\normalfont\ttfamily}{\mathtt}
\DeclareOldFontCommand{\bf}{\normalfont\bfseries}{\mathbf}
\DeclareOldFontCommand{\it}{\normalfont\itshape}{\mathit}
\DeclareOldFontCommand{\sl}{\normalfont\slshape}{\@nomath\sl}
\DeclareOldFontCommand{\sc}{\normalfont\scshape}{\@nomath\sc}
\makeatother

\usepackage{csquotes}  % For \enquotes
\newcommand\q{\enquote}

\usepackage{physics}   %For \curl

\newcommand{\ccat}[3]{{#1\, \underset{#3}{\lozenge}\,{#2}}}

\newcommand{\tm}{\times}

\newcommand \eps {\varepsilon}

\newcommand \N   {\mathbb{N}}
\newcommand \R   {\mathbb{R}}

\newcommand \K   {\mathcal{K}}
\newcommand \Kinf{\mathcal{K_\infty}}

\newcommand \KL  {\mathcal{KL}}

\newcommand \LL  {\mathcal{L}}

\newcommand \PD   {\mathcal{P}}

\newcommand{\T}{\ensuremath{\mathcal{T}}}  % E.g. for semigroups
\newcommand{\Uc}{\ensuremath{\mathcal{U}}}

\newcommand{\Vc}{\ensuremath{\mathcal{V}}}

\newcommand{\Sc}{\ensuremath{\mathcal{S}}}

\newcommand{\vertiii}[1]{{\left\vert\kern-0.25ex\left\vert\kern-0.25ex\left\vert #1 
    \right\vert\kern-0.25ex\right\vert\kern-0.25ex\right\vert}}

\newcommand \qrq   {\quad\Rightarrow\quad}

\newcommand \Iff   {\Leftrightarrow}

\newcommand{\normt}[1]{{\left\vert\kern-0.25ex\left\vert\kern-0.25ex\left\vert #1 
		\right\vert\kern-0.25ex\right\vert\kern-0.25ex\right\vert}}

\newcommand \SSet   {\mathcal{S}}
\newcommand{\Simp}{\ensuremath{\mathcal{S}}}

\newcommand{\clo}[1]{\overline{#1}}

\newcommand{\mir}[1]{{\color{red}\bf AM: #1}}     % Comments, problems, suggestions etc. by Andrii Mironchenko
\newif\ifMath					%For mathematical applications only
\newif\ifEngi					%For engineering applications only

\newif\ifDFGtext					 %DFG explanations to the proposals, not a part of the final proposal.

\newif\ifAndo              %Some comments, that are not intended for final version. 
\newif\ifExercises					%Yes, if the exercises are included (e.g. exercises do not fit for Habil, or PhD thesis)
\newif\ifSolutions          %Yes, if the solutions (in this if environment) should be included.
\newif\ifGerman							%Yes, if the German parts should be included
\newif\ifEnglish						%Yes, if the English parts should be included

\newif\ifnothabil						%Yes, if this is NOT a habilitation thesis
\newif\ifFuture							%Yes, if some 'Future', unfinished parts should be added

\newif\ifConf                    %Parts only for a conference version of the paper
\newif\ifJournal								 %Parts only for a journal version of the paper

\newif\ifNOTFORBOOK
\newif\ifFullVersion
\newif\ifExludedDueToSpaceReasons

\usepackage{xifthen}

\newcommand{\einsnorm}[2]{\ensuremath{
    \!\!\;\!\!\!\;
    \left\bracevert\!\!\!\!\!\left\bracevert
    \!
		\ifthenelse{\isempty{#2}}{#1}{#1(#2)}
    \!
      \right\bracevert\!\!\!\!\!\right\bracevert
    \!\!\;\!\!\!\;
  }}

\usepackage{xcolor}
\definecolor{blond}{rgb}{0.98, 0.94, 0.75}
	
\newlength\mytemplen
\newsavebox\mytempbox

\makeatletter
\newcommand\mybluebox{%
    \@ifnextchar[%]
       {\@mybluebox}%
       {\@mybluebox[0pt]}}

\def\@mybluebox[#1]{%
    \@ifnextchar[%]
       {\@@mybluebox[#1]}%
       {\@@mybluebox[#1][0pt]}}

\def\@@mybluebox[#1][#2]#3{
    \sbox\mytempbox{#3}%
    \mytemplen\ht\mytempbox
    \advance\mytemplen #1\relax
    \ht\mytempbox\mytemplen
    \mytemplen\dp\mytempbox
    \advance\mytemplen #2\relax
    \dp\mytempbox\mytemplen
    \colorbox{blond}{\hspace{1em}\usebox{\mytempbox}\hspace{1em}}}

\makeatother
\makeatletter
\let\origd=\d
\renewcommand*\d{
  \relax\ifmmode
    \mathrm{d}%
  \else
    \expandafter\origd
  \fi
}\makeatother

\usepackage{mathtools}

\makeatletter % changes the catcode of @ to 11
\newcommand{\pushright}[1]{\ifmeasuring@#1\else\omit\hfill$\displaystyle#1$\fi\ignorespaces}
\newcommand{\pushleft}[1]{\ifmeasuring@#1\else\omit$\displaystyle#1$\hfill\fi\ignorespaces}
\makeatother % changes the catcode of @ back to 12

\newcounter{syscounter}
\newenvironment{sysnum}{\begin{list}{($\Sigma{\arabic{syscounter}}$)}%
{\settowidth{\labelwidth}{($\Sigma4$)}
\settowidth{\leftmargin}{($\Sigma4$)~}%
\usecounter{syscounter}}}
{\end{list}}

\newcounter{WPcounter}%[section] %[section]
\newcounter{PRcounter}%[section] %[section]
\begin{document}
%\title{Live systems of varying dimension: modeling and stability}
\title{Modeling and stability analysis of live systems with time-varying dimension}

\author{Andrii Mironchenko, \IEEEmembership{Senior Member, IEEE}
\thanks{
The work of A.~Mironchenko is supported by the German Research Foundation (DFG) through the grants MI 1886/2-2 and MI 1886/3-1.
}
\thanks{A. Mironchenko is with 
Department of Mathematics, University of Bayreuth,
95447 Bayreuth, Germany.
Email: andrii.mironchenko@uni-bayreuth.de %Corresponding author.
}
}

% make the title area
\maketitle

%\IEEEpeerreviewmaketitle

\begin{abstract}
A major limitation of the classical control theory is the assumption that the state space and its dimension do not change with time. This prevents analyzing and even formalizing the stability and control problems for open multi-agent systems whose agents may enter or leave the network, industrial processes where the sensors or actuators may be exchanged frequently, smart grids, etc.
In this work, we propose a framework of live systems that covers a rather general class of systems with a time-varying state space. We argue that input-to-state stability is a proper stability notion for this class of systems, and many of the classic tools and results, such as Lyapunov methods and superposition theorems, can be extended to this setting.
\end{abstract}

\begin{IEEEkeywords}
%Not all are from the IEEE List.
Nonlinear systems, modeling, infinite-dimensional systems, input-to-state stability, Lyapunov methods, impulsive systems, multi-agent systems
\end{IEEEkeywords}

\section{Introduction}

Systems theory constitutes a powerful paradigm for the analysis and control of linear and nonlinear systems that can handle the lack of information about the system, acting disturbances, communication constraints, and many further obstructions on the way to the practical implementation of the controllers. 
This astonishing progress was achieved under a foundational structural assumption that goes through the whole body of the mathematical systems theory: the state space does not change in time.

%Individual organisms, being looked upon closely enough, are not described by a fixed number of parameters.
%For example, plants are constructed of repeated units ('modules', 'compartments'), and during growth, a plant 
%%intermittently differentiates lateral organ primordia, which may become leaves, flowers, or bracts.
%intermittently creates new organs: leaves, flowers, or bracts. If we model each organ as a finite-dimensional dynamical system, we see that a plant is a system whose dimension grows in time.

However, this assumption is frequently not fulfilled in natural and human-made systems.
The number of individuals in the populations of organisms changes in time due to the birth and decay processes. 
Plants (considered as a system consisting of repeated units) intermittently create new organs such as leaves, flowers, or bracts. 
Dynamical systems with a variable state space, which we will call \emph{live systems}, appear naturally in control applications.  
An archetypal problem of this kind is designing the optimized adaptive traffic control system so that the controller is viable and scalable even though new roads and cars are entering or leaving the network \cite[Section 4.1]{LAE17}.
A related problem is the organization of a smart grid that ensure robust and effective generation and transport of energy even though the size and topology of the network change due to the attachment and detachment of the microgrids 
%(networks composed of distributed generation, storage and load elements) 
to the system \cite{DSB14, FMX11}. 
Another recent problem of this kind is a consensus of multi-agent systems with new agents entering or leaving the network (open multi-agent systems, OMAS) \cite{PiS13,HeM16,FrF21,GaH22,Viz22}.
In social networks, the users may enter and leave communities modifying the group behavior \cite{GrJ12}. 
Last but not least, most large-scale industrial processes are live systems, which are updated or modified from time to time: new sensors may be added, certain actuators can be exchanged, components may become malfunctioning, and subparts of the network may be added or excluded from the system, etc. \cite{Sto09}.

Most current control designs do not take into account the variable dimension of the state space, and thus they are developed only for a model of a system that is valid for a relatively short period. As argued in \cite{Sto09}, the control designs are also predominantly monolithic. That is, after a change of the model of a system, one has to develop a new controller from scratch, which is a costly and time-consuming process.

In computer science, live systems have been investigated in the context of networks of software agents. See \cite{PiS13} for an overview, \cite{MiS18} for results in the context of dynamic networks, \cite{AAF08} in the area of so-called self-stabilizing decentralized algorithms, and \cite{DFG06} in the context of fault-tolerant population protocols. 
The term OMAS seems to have origins in computer science as well; see \cite{MaS01}. Dynamical communities have also been studied within the game theory \cite{LST16}. Statistical methods have been used for the analysis of OMAS in \cite{HeM17}, and simulations were employed to study the dynamic average consensus problem in \cite{ZhM10}. 
Somewhat related (but different) concept of multi-agent systems with variable state space are evolutionary multi-agent systems \cite{BDS15}.

\textbf{Challenges.} Although live systems are omnipresent, few works are devoted to their analysis within the control and dynamical systems theory. 
This is due to several conceptual problems in the modeling of live systems. 
First of all, one needs to go beyond the classical concept of the state space. 
Secondly, the stability concepts need to be revisited. Indeed, if new agents with a magnitude of the state bounded away from zero may enter the network at arbitrarily large times, there is no hope that the reasonably defined state of the system approaches zero (or any other point). Thus, the classical attractivity and asymptotic stability properties may fail to be the proper stability concepts for control systems.

\textbf{Approaches to live systems in control theory.} 
Particular questions in the theory of live systems have been addressed in systems theory. 
The adaptive control theory for nonlinear systems with unknown parameters \cite{KKK95, AsW08} allows us to handle the problem of updating the existing sensors or actuators with similar dynamics but distinct parameters. Yet, the problem of adding new sensors or actuators is outside of the scope of adaptive control. 
Fault-tolerant and reconfigurable control \cite{ZhJ08, BKL06, Pat97} considers the design of controllers that are robust with respect to malfunctions of some of its components (sensors, actuators, observers). However, most of the research is devoted to analyzing some particular types of failures that may occur.
%If the state space is infinite-dimensional, then the state space may be continuously changing. For example, this is the case in that study of PDEs with moving boundaries (Stefan's problem). 
Plug-and-play control envisaged in \cite{Sto09} aims to develop self-configuring control methods that will remain viable if new sensors or controllers are added.

%One of the ways to study such systems with a variable number of components is to model such systems via partial differential equations (PDEs),
%governing some density function of the individuals (particles, agents) in a certain area.
%%Such are the PDEs arising in mathematical ecology, equations arising by mean-field approximation, etc. 
%PDE models provide reasonable approximations of the real behavior of the systems only if the number of components is large enough and the components are similar, which is not always the case.

If the maximal number of components is known, one can study the dynamics of such a network in the largest possible state space. The entering or leaving of subsystems to/from the network can thus be modeled by switches in the system's structure \cite{VMN18}. 
In \cite{Ver12}, the concept of pseudo-continuous multi-dimensional multi-mode systems has been introduced that can be understood as a switched system with finitely many modes of a distinct dimension. The change of a dimension is modeled by a switch into the other mode. 
Both in \cite{VMN18} and \cite{Ver12}, the maximal dimension of the system is finite and known in advance.

In some cases, it is possible to overapproximate a large but finite network of a possibly unknown size via an infinite network and transfer the stability results for the infinite overapproximation to any possible subnetwork \cite{NMW22,Mir21}. 

In \cite{FrF21}, the authors consider open discrete-time multi-agent systems. The authors do not assume any a priori upper bound for the dimension of the network. Thus the dimension of the system may well converge to infinity with the time, though it remains finite at any moment. 
The asymptotic stability in \cite{FrF21} is studied in weighted norms, given by the usual norms of finite-dimensional vectors divided by a square root of the dimension of the vector. 
The stability in such weighted norms may be interesting for some applications but can be non-typical for other ones. 

\textbf{Contribution.} 
%The main contribution of this paper is to formulate the Plug-and-Play control problem in a quite general setting.
The main contribution of this paper is to propose a new framework for modeling and stability analysis of live systems. 
We consider a live system as an impulsive system that changes at impulse times not only the value of the state but also the configuration of the system. Each configuration of a live system is a control system in a classical sense, characterized by the state space, input space, and the flow map \cite{KFA69,Wil72,Son98a}. 
Our setting is very general, including open multi-agent systems, systems with unknown dimension, switched systems of variable dimension, etc.

We show that live systems retain the essential features of control systems.
Although the set of all states does not have a linear structure (in particular, the dimension of the state may vary with time), it is possible to introduce a kind of pseudonorm on this set. 

We treat an arrival of an agent as an input to the system.  
%In fact, it is not a big difference whether a new agent comes into the network or the state of one of the existing agents is changed in an impulsive way.
This motivates us to analyze the stability of live systems in the sense of input-to-state stability (ISS) introduced in \cite{Son89}, and extended to impulsive systems in \cite{HLT08,DaM13b}. See also \cite{Son08, MiP20} for the survey, and \cite{Mir23} for systematic development of the theory.
It turns out that many characterizations for ISS of infinite-dimensional systems in Banach spaces shown in \cite{MiW18b} are still valid for live systems, as the continuity of trajectories and the specific properties of Banach spaces are not used in the arguments in \cite{MiW18b}. 
Finally, we introduce the concept of ISS Lyapunov functions for live systems based on the corresponding concept from the ISS theory of impulsive systems \cite{HLT08,DaM13b}. And as in impulsive systems, having an ISS Lyapunov function, we can prove the ISS of the live system under certain dwell-time conditions imposed on the density of impulses.

%Overall, although live systems have several rather \q{non-traditional} traits as the time-varying dimension of the state, and discontinuity of trajectories, the live systems still retain the key properties of control systems, and a great deal of the stability theory understood in the ISS sense can be transferred to this class of systems. 
%And all this holds in a very general setting, including open multi-agent systems, systems of unknown dimensions, switched systems of variable dimensions, etc. 

%We are sure that control systems with a variable state space will serve as a basis for the modeling of reproduction processes, for modeling of the interaction of populations with the environment, and will find the applications in the paradigm of plug-and-play control \cite{Sto09}.

\textbf{Notation.} 
For two sets $X,Y$ denote by $C(X,Y)$ the linear space of continuous functions, mapping $X$ to $Y$.

For the formulation of stability properties, the following classes of comparison functions are useful:
%\begin{equation*}
%\begin{array}{ll}
%{\K} &:= \left\{\gamma:\R_+ \to \R_+ \left|\ \right. \gamma\mbox{ is continuous and strictly increasing, }\gamma(0)=0\right\}\\
%{\K_{\infty}}&:=\left\{\gamma\in\K\left|\ \gamma\mbox{ is unbounded}\right.\right\}\\
%{\LL}&:=\left\{\gamma:\R_+ \to \R_+ \left|\ \gamma\mbox{ is continuous and decreasing with}\right.
 %\lim\limits_{t\rightarrow\infty}\gamma(t)=0 \right\}\\
%{\KL} &:= \left\{\beta: \R_+^2 \to \R_+ \left|\ \right. \beta(\cdot,t)\in{\K},\ \forall t \geq 0,\  \beta(r,\cdot)\in {\LL},\ \forall r >0\right\}
%\end{array}
%\end{equation*}
\begin{equation*}
\begin{array}{ll}
{\K} &:= \left\{\gamma:\R_+\rightarrow\R_+\left|\ \gamma\mbox{ is continuous, strictly} \right. \right. \\
&\phantom{aaaaaaaaaaaaaaaaaaa}\left. \mbox{ increasing and } \gamma(0)=0 \right\}, \\
{\K_{\infty}}&:=\left\{\gamma\in\K\left|\ \gamma\mbox{ is unbounded}\right.\right\},\\
{\LL}&:=\left\{\gamma:\R_+\rightarrow\R_+\left|\ \gamma\mbox{ is continuous and strictly}\right.\right.\\
&\phantom{aaaaaaaaaaaaaaaa} \text{decreasing with } \lim\limits_{t\rightarrow\infty}\gamma(t)=0\},\\
{\KL} &:= \left\{\beta:\R_+\times\R_+\rightarrow\R_+\left|\ \beta \mbox{ is continuous,}\right.\right.\\
&\phantom{aaaaaa}\left.\beta(\cdot,t)\in{\K},\ \beta(r,\cdot)\in {\LL},\ \forall t\geq 0,\ \forall r >0\right\}. \\
\end{array}
\end{equation*}

%For the formulation of stability properties, we use the standard classes of comparison functions:
%\index{comparison!functions}
%\index{class!$\K$}
%\index{class!$\LL$}
%\index{class!$\KL$}
%\index{class!$\PD$}
%\index{class!$\Kinf$}
%\index{positive-definite function}
%\index{function!positive definite}
%\begin{equation*}
%\begin{array}{ll}
%{\PD} &:= \left\{\gamma \in C(\R_+): \gamma(0)=0 \mbox{ and } \gamma(r)>0 \mbox{ for } r>0  \right\} \\
%{\K} &:= \left\{\gamma \in \PD : \gamma \mbox{ is strictly increasing}   \right\}\\
%%{\K} &:= \left\{\gamma:\R_+\rightarrow\R_+: \gamma\mbox{ is continuous, }\gamma(0)=0\right.\\
%%&\phantom{aaaaaaaaaaaaaaaaaaa}\left.\text{and strictly increasing}\right\}\\
%{\K_{\infty}}&:=\left\{\gamma\in\K: \gamma\mbox{ is unbounded}\right\}\\
%{\LL}&:=\{\gamma\in C(\R_+): \gamma\mbox{ is strictly decreasing with } \lim\limits_{t\rightarrow\infty}\gamma(t)=0 \}\\
%{\KL} &:= \left\{\beta \in C(\R_+\times\R_+,\R_+): \beta(\cdot,t)\in{\K},\ \forall t \geq 0, \  \beta(r,\cdot)\in {\LL},\ \forall r > 0\right\}
%\end{array}
%\end{equation*}
%Functions of class $\PD$ are also called \textit{positive definite functions}.
An up-to-date compendium of results concerning comparison functions can be found in \cite{Kel14} and \cite[Appendix]{Mir23}.

\section{Live systems}

As we strive to develop a unified framework of live systems, we start with a rather general yet classical concept of a control system, whose variations have been used in control theory at least since the 1960s \cite{KFA69, Wil72, Son98a}.

\begin{definition}
\label{Steurungssystem}
Consider the triple $\Sigma=(X,\Uc,\phi)$ consisting of 
\index{state space}
\index{space of input values}
\index{input space}
\begin{enumerate}[label=(\roman*)]  
    \item A set $X$ called the \emph{state set}.
    \item An \emph{input set} $\Uc \subset \{u:\R_+ \to U\}$, 
					where $U$ is a set called \emph{set of input values}.
					
We assume that the following two axioms hold:
                    
\emph{The axiom of shift invariance}: for all $u \in \Uc$ and all $\tau\geq0$ the time-shifted input $u(\cdot + \tau)$ belongs to $\Uc$. 

\emph{The axiom of concatenation}: for all $u_1,u_2 \in \Uc$ and all $t>0$ the \emph{concatenation of $u_1$ and $u_2$ at time $t$}, defined by
\begin{equation}
\big(\ccat{u_1}{u_2}{t}\big)(\tau):=
\begin{cases}
u_1(\tau), & \text{ if } \tau \in [0,t], \\ 
u_2(\tau-t),  & \text{ if } \tau > t,
\end{cases}
\label{eq:Composed_Input}
\end{equation}
belongs to $\Uc$.

    \item A map $\phi:D_{\phi} \to X$, with a domain of definition $D_{\phi}\subseteq \R_+ \times X \times \Uc$ (called \emph{transition map}), such that for all $(x,u)\in X \tm \Uc$ it holds that $D_{\phi} \cap \big(\R_+ \times \{(x,u)\}\big) = [0,t_m)\tm \{(x,u)\} \subset D_{\phi}$, for a certain $t_m=t_m(x,u)\in (0,+\infty]$.
		
		The corresponding interval $[0,t_m)$ is called the \emph{maximal domain of definition} of $t\mapsto \phi(t,x,u)$.
		
\end{enumerate}
The triple $\Sigma$ is called a \emph{(control) system}, if the following properties hold:
\index{property!identity}

\begin{sysnum}
    \item\label{axiom:Identity} \emph{The identity property:} for every $(x,u) \in X \times \Uc$
          it holds that $\phi(0, x,u)=x$.
\index{causality}
    \item \emph{Causality:} for every $(t,x,u) \in D_\phi$, for every $\tilde{u} \in \Uc$, such that $u(s) =
          \tilde{u}(s)$ for all $s \in [0,t]$ it holds that $[0,t]\tm \{(x,\tilde{u})\} \subset D_\phi$ and $\phi(t,x,u) = \phi(t,x,\tilde{u})$.
    %\item \label{axiom:Continuity} \emph{Continuity:} for each $(x,u) \in X \times \Uc$ the map $t \mapsto \phi(t,x,u)$ is continuous on its maximal domain of definition.
\index{property!cocycle}
        \item \label{axiom:Cocycle} \emph{The cocycle property:} for all
                  $x \in X$, $u \in \Uc$, for all $t,h \geq 0$ so that $[0,t+h]\tm \{(x,u)\} \subset D_{\phi}$, we have
\[
\phi\big(h,\phi(t,x,u),u(t+\cdot)\big)=\phi(t+h,x,u).
\]
\end{sysnum}
If $t_m(x,u)=\infty$ for all $x\in X$ and $u\in\Uc$, we call $\Sigma$ \emph{forward complete}.
\end{definition}

Definition~\ref{Steurungssystem} comprises the most basic and essential properties of control systems. Notably, $X$ and $U$ are merely sets, and we do not require any further structure (linearity, topology, norm, metric, etc.) from these sets. 
This is a typical feature in the early references \cite{KFA69,Wil72,Son98a}. 
The notion of a \emph{state set} (in contrast to state space) we adopt from \cite[Definition 1.1, p. 5]{KFA69}.

Fix any set $S$, which we call \emph{configuration set}. We call its elements \emph{configurations}. 
With each configuration $Q \in S$, we associate a control system $\Sigma_Q:=(X_Q,\Uc,\phi_Q)$.
To avoid the notational complications, we assume that the set of inputs is the same for all systems $\Sigma_Q$, $Q \in S$.

\ifAndo
\mir{There was a comment of a reviewer whether assuming the same input set is a reasonable assumption.}
\fi

\begin{remark}
\label{rem:Labeling} 
In this work, for any $Q \in  S$, we identify any element $y \in X_Q$ with a labeled pair $(Q,y)$. 
However, we drop the labels to simplify the notation.
\end{remark}

We will define a live system $\Sigma$ as a system induced by the family $(\Sigma_Q)_{Q \in  S}$, whose configuration may change with time.
We denote the configuration of $\Sigma$ at time $t$ by $I(t) \in  S$.

\begin{ass}
\label{ass:set of state spaces} 
The changes of the state space or the impulsive changes of the system's state occur at certain time instants given by the increasing infinite sequence $\T:=(t_k)_{k\in\N}$ without accumulation points. We call $\T$ \emph{impulse time sequence}.

Furthermore, the configuration $I(\cdot)$ remains constant between the impulse times, that is $I(t) = I(t_k)$ for $t\in[t_k,t_{k+1})$.
\end{ass}

The transitions between configurations are given by the map $q: S \tm \N \to  S$. 
The change of the configuration of the system $\Sigma$ at impulse times $t_k$ is given by
\begin{eqnarray}
I(t_k):=q(I(t_{k}^-),k),\quad k\in\N.
\label{eq:Configuration-change}
\end{eqnarray}

As the configuration changes, a system trajectory with an initial condition in $X_Q$ leaves $X_Q$ and jumps to the state set of the system in another configuration. This motivates us to consider the following set as the state set for the live system composed of $(\Sigma_Q)_{Q \in  S}$:
\begin{eqnarray}
X:=\bigcup_{Q\subset  S} X_Q.
\label{eq:state-set}
\end{eqnarray}
In view of Remark~\ref{rem:Labeling}, each element of $X$ \q{knows} its configuration. Thus, the union 
in \eqref{eq:state-set} is disjoint in the sense that for any $x \in X$ there is a unique $Q \in  S$ such that $x \in X_Q$. 

From now on, we assume that $\Uc$ is the space of piecewise continuous functions from $\R_+$ to $U$ (with this assumption, the values of inputs are well-defined at the impulse times).

We fix a sequence of impulse times $\T = (t_k)_{k=1}^\infty$, and construct the flow map $\phi$ of a live system $\Sigma$.

Pick any initial configuration $I_0 \in  S$, any initial condition $x \in X_{I_0}$, any sequence of impulse times $\T$, any input $u\in \Uc$. 
As long as the system stays in the initial configuration, we define the dynamics of $\Sigma$ by 
\[
\phi(t,x,u):=\phi_{I_0}(t,x,u),\quad t \in [0, \min\{t_{m,I_0}(x,u),t_1\}),
\]
where $t_{m,I_0}(x,u)$ is the maximal existence time of $\phi_{I_0}(\cdot,x,u) \subset \Sigma_{I_0}$.
If $t_{m,I_0}(x,u) < t_1$, then $t_m(x,u):=t_{m,I_0}(x,u)$, and $\phi(\cdot,x,u)$ is constructed.

Otherwise, define
\[
\phi(t_1,x,u):=g_{I(t_1^-)I(t_1)}(\phi(t_1^-,x,u),u(t_1^-)),
\]
where $g_{Q_1Q_2}:X_{Q_1} \tm U \to X_{Q_2}$ describes the jump of an element from the state set of a configuration $Q_1$ 
to the configuration $Q_2$.
Now the system $\Sigma$ is in a new configuration $I(t_1)$, and its evolution is governed by the flow $\phi_{I(t_1)}$. 
%
%It is also well-posed by assumption, and there is a unique solution 
%$\hat\phi(\cdot,t_1,\phi(t_1,x,u),u(\cdot+t_1))$ defined on some nonempty subinterval $[t_1,\tau_2)$ of $[t_1,t_2)$. If $\tau_2 <t_2$, then we set $t_m(x,u):=\tau_2$ and define 
\[
\phi(t,x,u):=\phi_{I(t_1)}\big(t-t_1,\phi(t_1,x,u),u(\cdot+t_1)\big), 
\]
for $t\in [t_1,t_1+\min\{t_{m,I(t_1)}(\phi(t_1,x,u),u(\cdot+t_1)),t_2-t_1\})$.

Here $u(\cdot+t_1) \in\Uc$ due to the axiom of shift-invariance. Note that the map $\phi$ also depends on the sequence $\T$. However, we do not explicitly express it in our notation.

Doing this procedure repeatedly, we obtain the flow map
\[
\phi:D_{\phi} \to X, \quad D_{\phi}\subseteq \R_+ \times X \times \Uc,
\]
such that for all $(x,u)\in X \tm \Uc$ there is $t_m=t_m(x,u)\in (0,+\infty]$ such that
\[
D_{\phi} \cap \big(\R_+ \times \{(x,u)\}\big) = [0,t_m)\tm \{(x,u)\} \subset D_{\phi}.
\]

We immediately see that
\begin{proposition}
\label{prop:Properties-Flow-map} 
For each impulse time sequence $\T$, the corresponding triple $\Sigma^\T:=(X,\Uc,\phi)$ is a control system.
\end{proposition}

\begin{proof}
All properties follow from the fact that the flow maps for particular configurations satisfy these properties (as ODE systems) and from the construction of the flow $\phi$.
\end{proof}

To introduce the stability concepts of live systems, we need to measure distances. 
Thus, from now on, we assume that all $X_Q$ are normed vector spaces endowed with the corresponding norms $\|\cdot\|_{X_Q}$.
At the same time, $X$ is still only a set. 

However, we can introduce the following concept of pseudonorm
\begin{definition}
\label{def:pseudonorm-on-Xc} 
We introduce the map $\|\cdot\|_{X}:X \to \R_+$ as
\begin{align}
\label{eq:pseudonorm-on-Xc} 
\|x\|_{X}:= \|x\|_{X_Q},\quad x\in X_Q, \ Q \in  S.
\end{align}
%\[
%\|x\|_{X}:= \|x\|_{X_Q} = d(x,0_{Q}),\quad x\in X_Q, \ Q \in  S.
%\]
We call this map (by abuse of terminology) a \emph{pseudonorm} on $X$.
\end{definition}

As $X$ does not possess a linear structure, there is no triangle property for the map $\|\cdot\|_{X}$. 
However, we still have the following:
\begin{proposition}
\label{prop:pseudonorm-properties} 
The map $\|\cdot\|_{X}$ satisfies the following properties:
\begin{itemize}
	\item $\|x\|_{X}\geq 0$ for all $x \in X$.
	\item $\|x\|_{X}= 0$ if and only if $x$ is a zero element in $X_Q$ for some $Q\in  S$.
	\item For each $x \in X$ and any $\lambda \in\R$ it holds that $\lambda x \in X$, and 
\begin{eqnarray}
\|\lambda x \|_{X} = |\lambda| \| x \|_{X}.
\label{eq:Homogeneity}
\end{eqnarray}
\end{itemize}
\end{proposition}

\begin{proof}
The first two properties are evident. To see the latter fact, note that for any $x \in X$, there is $Q \in  S$ such that 
$x \in X_Q$. As $X_Q$ is a linear space, $\lambda x \in X_Q \subset X$, and 
\[
\|\lambda x\|_{X} = \|\lambda x\|_{X_Q} = |\lambda| \| x\|_{X_Q}= |\lambda| \| x\|_{X}.
\]
\end{proof}

\section{Stability of live systems and its characterization}

As new agents with a state of magnitude bounded away from zero may enter the system at arbitrarily large times, a system cannot be stable or attractive in the classical sense. At the same time, as we view the arrival of new subsystems as the inputs to our system, it is reasonable to use the concept of input-to-state stability to study the asymptotics of such systems.

From now on, we assume that $\Uc \subset \{u:\R_+ \to U\}$ is a normed vector \emph{space of inputs} 
endowed with a norm $\|\cdot\|_{\Uc}$, where $U$ is a normed vector \emph{space of input values}.
As before, we assume that the $\Uc$ satisfies \emph{the axiom of shift invariance}: for all $u \in \Uc$ and all $\tau\geq0$ the time
shift $u(\cdot + \tau)$ belongs to $\Uc$. But now in addition we assume that \mbox{$\|u\|_\Uc \geq \|u(\cdot + \tau)\|_\Uc$}.

As $X$ is endowed with a map $\|\cdot\|_{X}$ that has many properties of a norm, we can define ISS in a usual way:

\begin{definition}
\label{def:ISS}
\index{stability!input-to-state}
\index{input-to-state stability}
\index{ISS}
For a given sequence $\T$ of impulse times, we call a system $\Sigma^\T:=(X,\Uc,\phi)$ \emph{input-to-state stable (ISS)} if  it is forward complete and there exist $\beta\in\KL,\ \gamma\in\K_{\infty}$, such that for any $x \in X$, all $ u \in \Uc$, and all $t\geq 0$ it holds that
\begin{equation}
\label{eq:ISS}
\|\phi(t,x,u)\|_{X} \leq \beta(\|x\|_{X},t) + \gamma(\|u\|_{\Uc}).
%\|\phi(t,t_0,x,u)\|_{X} = \|\phi(t,t_0,x,u)\|_{X_{I(t)}} \leq \beta(\|x\|_{X_{I_0}},t-t_0) + \gamma(\|u\|_{\Uc}).
\end{equation}
A live system $\Sigma$ is called \emph{uniformly ISS} over a given set $\Simp$ of admissible sequences of impulse times if $\Sigma^\T$ is ISS for every $\T\in\Simp$, with $\beta$ and $\gamma$ independent of the choice of the sequence from the class $\Simp$.
\end{definition}

%Note that we do not fix the initial state space. We assume that the system may initially consist of an arbitrary finite number of subsystems.

Many properties of \q{classic} ISS infinite-dimensional systems (as defined in \cite{MiP20}) can be transferred to the general live systems.
%, since the continuity of trajectories and the Banach space structure of the state space do not play the essential role. 
%Thus, such properties of \q{classic} infinite-dimensional systems (as defined in \cite{MiP20}) can be transferred to the general live systems.

\begin{definition}
\label{def:CIUCS} 
\index{property!convergent input-uniformly convergent state (CIUCS)} 
For a given sequence $\T$ of impulse times, a forward complete system $\Sigma^\T:=(X,\Uc,\phi)$
%We say that a live system $\Sigma$  
has a
 \emph{convergent input-uniformly convergent state (CIUCS) property}, 
if for each 
$u\in\Uc$ such that $\lim_{t \to \infty}\|u(t+\cdot)\|_\Uc = 0$, and for any $r>0$, it holds that
\[
\lim_{t \to \infty}\sup_{\|x\|_{X}\leq r}\|\phi(t,x,u)\|_{X} = 0.
\] 
\end{definition}

We denote for short $B_{r,\Uc}:=\{u\in\Uc: \|u\|_{\Uc} \leq r\}$, and $B_{r}:=\{x\in X: \|x\|_{X} \leq r\}$. 

Live systems share the following basic property with input-to-state stable finite-dimensional systems.
\begin{proposition}
\label{prop:Converging_input_uniformly_converging_state}
Every input-to-state stable live system (for a fixed impulse time sequence) has the CIUCS property. 
\end{proposition}

\begin{proof}
Let $\Sigma$ be an ISS live system with a gain $\gamma \in\Kinf$.
Pick any $r>0$, any $u\in \Uc$ so that $\lim_{t \to \infty}\|u(\cdot + t)\|_{\Uc} = 0$, and any $\eps>0$. To show the claim of the proposition, we need to show that there is a time $t_\eps=t_\varepsilon(r,u)>0$ so that
\[
\|x\|_{X}\leq r,\quad t\geq t_\eps \qrq \|\phi(t,x,u)\|_{X}\leq\eps.
\]

Choose $t_1>0$ so that 
\begin{align}
\label{eq:Tail-condition}
\|u(\cdot + t)\|_{\Uc}\leq \gamma^{-1}(\eps),\quad t\geq t_1.
\end{align}

Due to the cocycle property and ISS of $\Sigma$, we have for any $x \in B_r$ that
\begin{eqnarray*}
\|\phi(t+t_1,x,u)\|_{X}  &=& \big\|\phi\big(t,\phi(t_1,x,u),u(\cdot + t_1)\big)\big\|_{X} \\
                           &\leq& \beta\big(\|\phi(t_1,x,u)\|_{X},t\big) + \gamma\big(\|u(\cdot+t_1)\|_\Uc\big) \\
                           &\leq& \beta\big(\|\phi(t_1,x,u)\|_{X},t\big) + \eps\\
                           &\leq& \beta\big(\beta(\|x\|_{X},t_1) + \gamma(\|u\|_\Uc),t\big) + \eps\\
                           &\leq& \beta\big(\beta(r,0) + \gamma(\|u\|_\Uc),t\big) + \eps.
\end{eqnarray*}
Pick any $t_2$ in a way that $\beta\big(\beta(r,0) + \gamma(\|u\|_\Uc),t_2\big) \leq \eps$.
This ensures that $\|\phi(t_2+t_1,x,u)\|_{X} \leq 2\eps.$
Using consequently the cocycle property, the ISS of $\Sigma$, and the estimate \eqref{eq:Tail-condition}, we obtain for all $t\geq0$:
\begin{align*}
\|\phi(t+t_2&+t_1,x,u)\|_X \\
&\leq \beta\big(\|\phi(t_2+t_1,x,u)\|_X,t\big) + \gamma\big(\|u(\cdot+t_1+t_2)\|_\Uc\big)\\
&\leq \beta\big(\|\phi(t_2+t_1,x,u)\|_X,t\big) + \eps\\
&\leq \beta\big(2\eps,0\big) + \eps.
\end{align*}
%
%Using ISS of $\Sigma$ combined with the cocycle property once again, we obtain for all $t\geq0$:
%\begin{align*}
%\|\phi(t+t_2&+t_1,x,u)\|_{X} \\
%&\leq \beta\big(\|\phi(t_2+t_1,x,u)\|_{X},t\big) + \gamma\big(\|u(\cdot+t_1+t_2)\|_\Uc\big).
%\end{align*}
%The estimate \eqref{eq:Tail-condition} ensures for all $t>0$ that
%\begin{eqnarray*}
%\|\phi(t+t_2+t_1,x,u)\|_{X} &\leq& \beta\big(\|\phi(t_2+t_1,x,u)\|_{X},t\big) + \eps\\
%&\leq& \beta\big(2\eps,0\big) + \eps.
%\end{eqnarray*}
Since $\eps>0$ can be chosen arbitrarily small, and since $\beta\big(2\eps,0\big) + \eps \to 0$ as $\eps\to 0$, the claim of the proposition follows.
\end{proof}

For a live system whose only inputs are due to the entering new agents, Propositon~\ref{prop:Converging_input_uniformly_converging_state} 
means that \emph{if a live system is ISS and the magnitude of the incoming agents tends to zero as time goes to infinity, then the system's state converges to zero}.

\section{Characterizations of ISS}

In this section, assuming that the impulse time sequence $\T$ is fixed (known a priori), we state a criterion for ISS of live systems in terms of weaker properties introduced next.

We denote for short $B_{r,\Uc}:=\{u\in\Uc: \|u\|_{\Uc} \leq r\}$, and $B_{r}:=\{x\in X: \|x\|_{X} \leq r\}$. 

\begin{definition}
\label{def:BRS}
\index{bounded reachability sets}
\index{BRS}
We say that \emph{$\Sigma^\T=(X,\Uc,\phi)$ has bounded reachability sets (BRS)}, if for any $C>0$ and any $\tau>0$ it holds that 
\[
\sup\big\{
\|\phi(t,x,u)\|_X : \|x\|_X\leq C,\ \|u\|_{\Uc} \leq C,\ t \in [0,\tau]\big\} < \infty.
\]
\end{definition}

\begin{definition}
\label{def:ULS}
A system $\Sigma^\T=(X,\Uc,\phi)$ is called \emph{uniformly locally stable (ULS)}, if there exist $ \sigma \in\Kinf$, $\gamma
          \in \Kinf \cup \{0\}$ and $r>0$ such that for all $ x \in \clo{B_r}$ and all 
					$u \in \clo{B_{r,\Uc}}$ the trajectory $\phi(\cdot,x,u)$ is defined on $\R_+$, and
\begin{equation}
\label{GSAbschaetzung}
\left\| \phi(t,x,u) \right\|_X \leq \sigma(\|x\|_X) + \gamma(\|u\|_{\Uc}) \quad \forall t \geq 0.
\end{equation}
\end{definition}

%\begin{definition}
%\label{def:UGS}
%A system $\Sigma=(X,\Uc,\phi)$ is called \emph{uniformly globally stable (UGS)}, if $\Sigma$ is forward complete and there exist $ \sigma \in\Kinf$, $\gamma \in \Kinf \cup \{0\}$ such that for all $ x \in X, u \in \Uc$ the estimate \eqref{GSAbschaetzung} holds.
%\end{definition}

%\begin{definition}
%\label{def:asymptotic_gain}
%\index{AG}
%\index{property!asymptotic gain}
%A forward complete system $\Sigma^\T=(X,\Uc,\phi)$ has the \emph{bounded input uniform asymptotic gain (bUAG) property}, if there
          %exists a
          %$ \gamma \in \Kinf \cup \{0\}$ such that for all $ \eps, r
          %>0$ there is a $ \tau=\tau(\eps,r) < \infty$ such
          %that
%\begin{equation*}
%%\label{eq:bUAG_Absch}
%\|u\|_{\Uc}\leq r,\ \ x \in B_{r},\ \ t \geq \tau \srs \|\phi(t,x,u)\|_X \leq \eps + \gamma(\|u\|_{\Uc}).
%\end{equation*}
%\end{definition}
%
%\begin{definition}
%\label{def:bULIM-property}
%We say that $\Sigma^\T=(X,\Uc,\phi)$ has the \emph{uniform limit property on bounded sets (bULIM)}, if there exists
    %$\gamma\in\Kinf\cup\{0\}$ so that for every $\eps>0$ and for every $r>0$ there is a $\tau = \tau(\eps,r)$ such that
%\begin{align}
%\|x\|_X &\leq r \ \wedge \ \|u\|_\Uc \leq r \nonumber\\
%& \qrq \exists t\leq\tau:\  \|\phi(t,x,u)\|_X \leq \eps + \gamma(\|u\|_{\Uc}).
%\label{eq:bULIM_ISS_section}
%\end{align}
%\end{definition}
%

We define the attractivity properties for systems with inputs.
\begin{definition}
\label{def:asymptotic_gain}
\index{AG}
\index{property!asymptotic gain}
A forward complete system $\Sigma=(X,\Uc,\phi)$ has the
\emph{uniform asymptotic gain (UAG) property}, if there
          exists a
          $ \gamma \in \Kinf \cup \{0\}$ such that for all $ \eps, r
          >0$ there is a $ \tau=\tau(\eps,r) < \infty$ such
          that
\begin{align}
\label{eq:bUAG_Absch}
\|u\|_{\Uc}\leq r,&\quad x \in B_{r},\quad t \geq \tau \nonumber\\
&\qrq \|\phi(t,x,u)\|_X \leq \eps + \gamma(\|u\|_{\Uc}).
\end{align}
\end{definition}

\begin{definition}
\label{def:Limit-properties}
We say that a forward complete live system $\Sigma^\T=(X,\Uc,\phi)$ has the \emph{uniform limit property (ULIM)}, if there exists
    $\gamma\in\Kinf\cup\{0\}$ so that for every $\eps>0$ and for every $r>0$ there is a $\tau = \tau(\eps,r)$ such that
\begin{align}
\|x\|_X &\leq r \ \wedge \ \|u\|_\Uc \leq r \nonumber\\
&\qrq \exists t\leq\tau:\  \|\phi(t,x,u)\|_X \leq \eps + \gamma(\|u\|_{\Uc}).
\label{eq:bULIM_ISS_section}
\end{align}
\end{definition}

%\begin{remark}
%%\label{rem:} 
%The properties that we have called UAG and ULIM are usually entitled bounded input uniform asymptotic gain property and bounded input uniform limit property respectively, while UAG and ULIM are reserved for more demanding concepts \cite{MiW18b}. 
%\end{remark}

We will need the following useful result:
\begin{proposition}
\label{prop:KL-fun-lemma-LSW} 
Let $g:\R_+\tm\R_+ \to\R$ satisfy the following properties:
\begin{enumerate}[label=(\roman*)]  
	\item For all $r>0$ and all $\varepsilon>0$ there is $\tau=\tau(\varepsilon,r)>0$ such that 
\begin{eqnarray}
s\leq r \ \wedge \ t\geq \tau  \qrq  g(s,t) \leq \varepsilon.
\label{eq:KL-fun-lemma-LSW-1}
\end{eqnarray}

	\item There is $\sigma_1 \in \Kinf$ and $\delta>0$, such that 
\begin{eqnarray}
s\leq \delta \ \wedge\  t\geq 0 \qrq g(s,t) \leq \sigma_1(s).
\label{eq:KL-fun-lemma-LSW-2}
\end{eqnarray}	

	\item $g$ is bounded on bounded sets.
\end{enumerate}
Then there is $\beta\in\KL$ such that 
\begin{eqnarray}
g(s,t) \leq \beta(s,t)\quad \forall s,t \in\R_+.
\label{eq:KL-fun-lemma-LSW-3}
\end{eqnarray}
\end{proposition}

\begin{remark}
\label{rem:Counterexample-for-ISW01} 
In \cite[Proposition 2.4]{ISW01} this result was stated without the assumption (iii). 
However, without the condition (iii) Proposition~\ref{prop:KL-fun-lemma-LSW} is not valid.
In particular, consider the function
\begin{eqnarray}
g(r,t) :=
\begin{cases}
\frac{r}{1-t}&,\text{if } r\ge 1, \ t\in[0,1),\\
re^{-t}&,\text{otherwise.} 
\end{cases}
\label{eq:counterex}
\end{eqnarray}
Clearly, this function satisfies the assumption (i) and (ii) of Proposition~\ref{prop:KL-fun-lemma-LSW}. However, the property (iii) does not hold, the function is unbounded on a compact set. In particular, $g$ cannot be bounded from above by a $\KL$-function.

The result in \cite[Proposition 2.5]{LSW96}, to which the authors in \cite{ISW01} refer to (and where a closely related statement was shown), is however correct, but instead of (ii) and (iii) the authors in \cite{LSW96} assume a stronger property that is reminiscent of uniform global stability.
\end{remark}

\begin{proof}
Take any $r>0$. By (i), there is $\tau = \tau(1,r)$ such that 
\begin{eqnarray}
s\leq r \ \wedge\  t\geq \tau \qrq g(s,t) \leq 1.
\label{eq:KL-fun-lemma-LSW-1a}
\end{eqnarray}
At the same, by (iii), for the above $r,\tau$ it holds that
\begin{eqnarray}
\sup_{s \leq r,\ t\leq \tau}g(s,t) <\infty.
\label{eq:KL-fun-lemma-LSW-BRS}
\end{eqnarray}		
Combining it with \eqref{eq:KL-fun-lemma-LSW-1a}, we see that 
\begin{eqnarray}
\sup_{s \leq r,\ t\geq 0}g(s,t) =: \xi_1(r) <\infty
\label{eq:KL-fun-lemma-LSW-2a}
\end{eqnarray}		
holds for some nondecreasing function $\xi_1$. 
Hence, there is $\xi\in\Kinf$ and $c>0$, such that 
\begin{eqnarray}
\sup_{s \leq r,\ t\geq 0}g(s,t) \leq \xi(r) + c,\quad r\geq 0.
\label{eq:KL-fun-lemma-LSW-UGB}
\end{eqnarray}		
The estimate \eqref{eq:KL-fun-lemma-LSW-UGB} together with the property (ii) ensures as in \cite[Proposition 2.43]{Mir23} that there is $\sigma\in\Kinf$ such that 
\begin{eqnarray}
\sup_{s \leq r,\ t\geq 0}g(s,t) \leq \sigma(r),\quad r\geq 0.
\label{eq:KL-fun-lemma-LSW-UGS}
\end{eqnarray}		
%Note that \eqref{eq:KL-fun-lemma-LSW-UGS} is stronger than a combination of (ii) and (iii). 
%It remains to show that \eqref{eq:KL-fun-lemma-LSW-UGS} in combination with (i) ensures the existence of a $\KL$-bound for $g$. 

We are going to construct a function $\beta \in \KL$ so that \eqref{eq:KL-fun-lemma-LSW-3} holds. 
Fix an arbitrary $r \in \R_+$.

Define $\eps_n:= 2^{-n}  \sigma(r)$, for $n \in \N$. The assumption (i) implies that there exists a sequence of times
$\tau_n:=\tau(\eps_n,r)$, which we may without loss of generality assume
to be strictly increasing, such that
\[
s \leq r \ \wedge \ t \geq \tau_n  \qrq  g(s,t) \leq \eps_n.
\]
From \eqref{eq:KL-fun-lemma-LSW-UGS}, we see that we may set $\tau_0 := 0$.
Define $\omega(r,\tau_n):=\eps_{n-1}$, for $n \in \N$ and $\omega(r,0):=2\eps_0=2\sigma(r)$.

%\begin{figure}[ht]
%\centering
%\begin{tikzpicture}[scale=0.8]
%\filldraw[gray!20] (0,0) rectangle (1.55,4);
%\filldraw[gray!20] (1.55,0) rectangle (3.925,2);
%\filldraw[gray!20] (3.925,0) rectangle (6,1);
%
%\draw[thick,->] (-1,0) -- (8,0);
%\draw[thick,->] (0,-1) -- (0,5);
%\draw[thick] (0.05,4) -- (-0.15,4) node[left] {$\eps_0$};
%\draw[thick] (0.05,2) -- (-0.15,2) node[left] {$\eps_1$};
%\draw[thick] (0.05,1) -- (-0.15,1) node[left] {$\eps_2$};
%\draw (0.5,4.7) to[bend left=15] (2,3.5) to[bend right=20] (4,2) to[bend left=20] (6,1) to[bend right=20] (8,0.25);
%
%\draw[dashed] (0.1,4) -- (1.55,4) -- (1.55,0) node[below] {$\tau_1$};
%\draw[dashed] (0.1,2) -- (3.925,2) -- (3.925,0) node[below] {$\tau_2$};
%\draw[dashed] (0.1,1) -- (6,1) -- (6,0) node[below] {$\tau_3$};
%
%\node (Omega) at (2.8,3.5) {$\omega(\delta,\cdot)$};
%\end{tikzpicture}
%\caption{Construction of the function $\omega(\delta,\cdot)$.}
%\label{Omega-Konstruktion}
%\end{figure}

Extend the definition of $\omega$ to a function $\omega(r,\cdot) \in \LL$.
We obtain for all $n=0,1,\ldots$ that
\[
s\leq r \ \wedge \ t \in (\tau_n,\tau_{n+1})  \qrq  g(s,t) \leq \eps_n < \omega(r,t).
\]
Doing this for all $r \in \R_+$, we obtain the definition of the function $\omega$.

Define $\hat \beta(r,t):=\sup_{0 \leq s \leq r}\omega(s,t) \geq
\omega(r,t)$ for $(r,t) \in \R_+^2$. From this definition, it follows that,
for each $t\geq 0$, $\hat\beta(\cdot,t)$ is
increasing in $r$ and $\hat\beta(r,\cdot)$ is decreasing in $t$ for each $r>0$ as
every $\omega(r,\cdot) \in \LL$.
Moreover, for each fixed $t\geq0$, $\hat \beta(r,t) \leq \sup_{0 \leq s \leq r}\omega(s,0)=2\sigma(r)$. This implies that $\hat\beta$ is continuous in the first argument at $r=0$ for any fixed $t\geq0$.

By \cite[Proposition 9]{MiW19b}, $\hat\beta$ can be upper bounded by certain $\tilde{\beta}\in \KL$, and the estimate
\eqref{eq:KL-fun-lemma-LSW-3} is satisfied with such a $\beta$.
\end{proof}

%The ISS superposition theorem from \cite{MiW18b} can be transferred to general live systems without significant changes in the formulation and the proof.
%, as the continuity of trajectories is not used in the proof in \cite{MiW18b}, and specific properties of the state space that go beyond those which the pseudometric space $X$ has are not used in \cite{MiW18b} either.

The ISS superposition theorem can be transferred to live systems without significant changes in the formulation:
\begin{theorem}[ISS superposition theorem]
\label{thm:UAG_equals_ULIM_plus_LS}
Let the impulse time sequence $\T$ be fixed, and $\Sigma^\T=(X,\Uc,\phi)$ be a forward complete live control system. The following statements are equivalent:
\begin{enumerate}[label=(\roman*)]
    \item\label{itm:ISS-Characterization-bounded-properties-1} $\Sigma^\T$ is ISS.
    %\item\label{itm:ISS-Characterization-bounded-properties-2} $\Sigma$ is UAG and UGS.
    %\item\label{itm:ISS-Characterization-bounded-properties-3} $\Sigma$ is bUAG and UGS.
	  \item\label{itm:ISS-Characterization-bounded-properties-4} $\Sigma^\T$ is UAG, CEP and BRS.
    \item\label{itm:ISS-Characterization-bounded-properties-5} $\Sigma^\T$ is ULIM, ULS and BRS.
    %\item\label{itm:ISS-Characterization-bounded-properties-6} $\Sigma$ is bULIM and UGS.
\end{enumerate}
\end{theorem}

\begin{proof}
The statement (i) easily implies both (ii) and (iii). 

Using the same arguments as in \cite{MiW18b}, one can see that both (ii) and (iii) imply that $\Sigma^\T$ is UAG, ULS and BRS. Assume that the gain $\gamma$ in UAG and ULS properties is the same. Let us show ISS of $\Sigma^\T$.

Define for all $r,\tau \geq 0$ the map
\begin{align*}
%\label{eq:}
g(r,\tau) &:= \sup\big\{ \|\phi(t,x,u)\|_X - \gamma(\|u\|_\Uc): \nonumber\\
&\qquad\qquad\qquad\qquad	 \|x\|_X\leq r,\ \|u\|_\Uc \leq r,\ t \geq \tau \big\}.
\end{align*}
We will show that this map satisfies the assumptions of Proposition~\ref{prop:KL-fun-lemma-LSW}. 

By UAG property, for all $\varepsilon>0$ and all $r>0$ there is $\tau=\tau(r,\varepsilon)>0$ such that for all 
$\|x\|_X\leq r$ and all $\|u\|_\Uc \leq r$ we have that 
\begin{align}
\label{eq:aux-UAG}
t\geq \tau \qrq \|\phi(t,x,u)\|_X \leq \varepsilon + \gamma(\|u\|_\Uc).
\end{align}
This shows that \eqref{eq:KL-fun-lemma-LSW-1} holds.
BRS of $\Sigma$ ensures that 
\[
\sup\big\{ \|\phi(t,x,u)\|_X : \|x\|_X\leq r,\ \|u\|_\Uc \leq r,\ t \in[0,\tau] \big\}<\infty.
\]
Combined with \eqref{eq:aux-UAG}, we see $g(r,\tau)$ is well-defined and finite for any $r,\tau$, and \eqref{eq:KL-fun-lemma-LSW-3} holds.

As $\Sigma^\T$ is ULS, there exists $\sigma\in\Kinf$ and $\delta>0$ such that whenever $\|x\|_X \leq\delta$ and $\|u\|_{\Uc} \leq \delta$, we have
\begin{equation*}
%\label{LS_Restatement}
t\geq 0 \qrq \|\phi(t,x,u)\|_X - \gamma(\|u\|_\Uc) \leq \sigma(\|x\|_X).
\end{equation*}

This shows the property \eqref{eq:KL-fun-lemma-LSW-2}. 
According to Proposition~\ref{prop:KL-fun-lemma-LSW}, there is $\beta\in\KL$ such that \eqref{eq:KL-fun-lemma-LSW-3} holds. 

Now pick any $x \in X$, any $u\in\Uc$ and any time $t\geq 0$.
Taking $r:=\max\{\|x\|_X,\|u\|_\Uc\}$, we obtain that 
\begin{align*}
%\label{eq:}
\|\phi(t,x,u)\|_X - \gamma(\|u\|_\Uc) &\leq g(\max\{\|x\|_X,\|u\|_\Uc\},t)\\
&\leq \beta(\max\{\|x\|_X,\|u\|_\Uc\},t).
\end{align*}
Thus:
\begin{align*}
%\label{eq:}
\|\phi(t,x,u)\|_X 
&\leq \beta(\|x\|_X,t) + \beta(\|u\|_\Uc,t) + \gamma(\|u\|_\Uc)\\
&\leq \beta(\|x\|_X,t) + \beta(\|u\|_\Uc,0) + \gamma(\|u\|_\Uc).
\end{align*}
This shows the ISS of $\Sigma^\T$.
\end{proof}

\section{Special classes of live systems}

To see the generality of our formalism, let us consider some special scenarios.

\subsection{Open multi-agent impulsive systems}
\label{sec:Open multi-agent systems}

		Open multi-agent systems are multi-agent systems with a time-varying number of agents, which may also grow to infinity with time. 
		Let $S_0$ be a (finite or infinite) set, which we call \emph{index set}, that labels all possible components of the network. 
Let for each $i\in S_0$ a Euclidean space $X_i$ endowed with the norm $\|\cdot\|_{X_i}$ be 
the state space of the $i$-th agent.

As a configuration set, we take a certain subset $ S$ of all finite nonempty subsets of $S_0$. 
Let $Q \in S$. In what follows, we denote by $(x_j)_{j \in Q}$ the vector consisting of $x_j \in X_j$ for $j\in Q$. 
Define the vector space
\begin{equation}
\label{eq:X_Q}
 \hspace{-2mm} X_Q := \bigtimes_{j\in Q}X_j.
\end{equation}
and endow it with a certain norm $\|\cdot\|_{X_Q}$ making $X_Q$ a (real) normed linear space. 

We define the state set $X$ for the OMAS $\Sigma$ by \eqref{eq:state-set}.

As before, for each time instant $t$, we denote the configuration of the system at time $t$ by $I(t) \in  S$.
The state space of the system at time $t$ is thus given by $X_{I(t)} \in X$.
Also, we take $\T$ as in Assumption~\ref{ass:set of state spaces}.

Let for all $t \in\R_+\backslash \T$ the dynamics of the system be given by the following equations: 
\begin{align}
\label{eq:Equations of motion}
\Sigma:\quad \dot{x}_i(t)	&= f_i(x(t),u_i(t)),\quad i \in I(t),\quad t>0.
\end{align}
Here $f_i:X \tm \R^{m_i} \to X_i$ is a map defined on the whole state set. 
We assume that $u_i \in L_\infty(\R_+,\R^{m_i})$, for all $i\in S_0$, and we define the total input to the system $\Sigma$ as 
$u:=(u_i)$, with $\|u\|_\Uc:=\sup_{i\in S_0}\|u_i\|_\infty$.
For the simplicity of notation, we also assume that all the signals $u_i$ are right-continuous.

\begin{remark}
\label{rem:f_i-dependence-configuration} 
As we implicitly label the states of $X_Q$ for any configuration $Q \in S$, $f_i$ depends both on the configuration of the system and the state of the system in this configuration.
\end{remark}

As for any $k\in\N$ and any $t \in (t_k,t_{k+1})$, it holds that $x(t) \in X_{I(t_k)}$, the dynamics of the system $\Sigma$
can be equivalently represented for $t\in (t_k,t_{k+1})$ by 
\begin{align}
\label{eq:Equations of motion-config}
\dot{x}(t)	&= f_{I(t_k)}(x(t),u(t)) :=(f_i(x(t),u_i(t)))_{i \in I(t_k)}.
\end{align}
Now $f_{I(t_k)}$ acts on $X_{I(t_k)} \tm \Uc$, and thus \eqref{eq:Equations of motion-config} is a \q{usual} ODE representing the dynamics of the overall system on $(t_k,t_{k+1})$. We denote the system $\Sigma$ in the configuration $I(t)$ by $\Sigma_{I(t)}$. The state space of this configuration is $X_{I(t)}$. 
We understand the solutions of \eqref{eq:Equations of motion-config} in the sense of Caratheodory. 

\begin{ass}
\label{ass:OMAS-well-posedness} 
We suppose that $\Sigma_Q$ is well-posed for any configuration $Q$. That is, for each initial condition $x_0 \in X_{I(t_k)}$ and any $u\in\Uc$, there exists a unique maximal solution of \eqref{eq:Equations of motion-config} with $x(t_k)=x_0$ on $[t_k,\tau)$ for some $\tau\in(t_k,t_{k+1})$. 
We denote this solution as $\hat\phi(\cdot,t_k,x_0,u)$.
\end{ass}

At impulse times, the configuration may change as some subsystems may leave the system, and others may enter the system. 
We denote by $D(t_k) \subset I(t_k^-)$ the set of indices of all subsystems that are leaving at time $t_k$, 
and by $B(t_k)$ the indices of  subsystems that enter the system at time $t_k$. 

\begin{ass}
\label{ass:OMAS-impulsive-times} 
We assume that 
\[
B(t_k) \cap I(t_k^-) = \emptyset.
\] 
That is, only subsystems which were not a part of the network immediately prior to time $t_k$ can join the system.

Thus, at impulse times $t_k$, $k\in\N$, we have that 
\[
I(t_k) = \big(I(t_k^-)\cup B(t_k)\big) \backslash D(t_k).
\]
We assume that $I(t_k) \in S$ for all $k\in\N$, which means that the configuration of the system remains admissible for all times.
\end{ass}

At impulse time $t_k$, the subsystems with indices belonging to $B(t_k)$ are entering the network $\Sigma$. We treat their initial conditions  as an input to the system:
\begin{align}
\label{eq:Impulsive-dynamics-new-states}
x_i(t_k) = u_i(t_k),\quad i \in B(t_k).
\end{align}
The states that remain in the network may jump instantaneously at time $t_k$:
\begin{align}
\label{eq:Impulsive-dynamics-old-states}
x_i(t_k) = g_i(x(t_k^-),u(t_k^-)),\quad i \in I(t_k^-)\backslash D(t_k).
\end{align}

We treat the initial states of new subsystems entering or leaving the network as an (impulsive) input to the system, similar to instantaneous changes in the states of subsystems within the network.

OMAS, as defined previously, constitute an important special case of live systems. On the other hand, this subclass is still quite general. In particular, if there is only one agent in the network, and thus the configuration does not change with time, we obtain a classical impulsive system. Other special cases of OMAS are systems with an unknown initial configuration and switched systems.

\subsection{Systems with an unknown initial configuration}
 
All our previous considerations are also valid for the empty time sequence $\T$. In this case, the system remains in the initial configuration for any initial condition and at the whole interval of existence. However, the initial configuration of the system is not fixed, and thus such a live system $\Sigma$ can be interpreted as a \q{classical} system with an unknown initial configuration.
Furthermore, $\Sigma$ is ISS if and only if each configuration $\Sigma_Q$, $Q\subset S_0$ is ISS, and the corresponding functions $\beta_Q$ and $\gamma_Q$ can be chosen uniformly in $Q$. 
If $S_0$ is a finite set, then clearly, $\Sigma$ is ISS iff each $\Sigma_Q$ is ISS.

\subsection{Switched systems} 

Let the index set $S_0$ be given. Let also $p\in\N$, and a collection of matrices $\{A_j \in \R^{p\tm p}:j\in S_0\}$ be given, together with a piecewise right-continuous signal $\sigma: \R_+\to S_0$, with at most an infinite number of discontinuities without accumulation points. A linear switched system is a system of the form
\[
\dot{x} = A_{\sigma(t)}x,\quad t>0.
\]
\emph{A switched system can be viewed as a live system with replacements and configurations of the same dimension}, where only one component is active at each instant, and the arrival of a new node coincides with the leaving of the old one. A switch is thus an impulse at which one component enters the system and another leaves the system.

Indeed, we can set $ S:=\{ \{j\}: j\in S_0\}$. 
Take $\T$ as a monotonically increasing to infinity sequence of discontinuities of $\sigma$. 
Then $I(0):=\sigma(0)$, $B(t_k):=\sigma(t_k)$, $D(t_k):=\sigma(t_k^-)$, and thus $I(t_k)=B(t_k)$.

\section{Lyapunov methods}

The concept of ISS Lyapunov functions can be naturally extended to live systems:
\begin{definition}
\index{ISS-Lyapunov function!impulsive systems}
\label{eq:ISS_Imp_LF}
A continuous function $V:X \to \R_+$ is called an \emph{ISS-Lyapunov function} for $\Sigma$ if $\exists \ \psi_1,\psi_2 \in \Kinf$, such that
\begin{equation}
\label{LF_ErsteEigenschaft}
\psi_1(\|x\|_X) \leq V(x) \leq \psi_2(\|x\|_X), \quad x \in X
\end{equation}
and $\exists \chi \in \Kinf$, $\alpha \in \PD$ and continuous function $\varphi:\R_+ \to \R$ with $\varphi(x)=0$ $\Iff$ $x=0$,  such that
$\forall x \in X$, $\forall \xi \in U$
%for all $x \in X$ and $u \in U$
it holds
\begin{equation}
\label{ISS_LF_Imp}
V(x)\geq\chi(\|\xi\|_U)\Rightarrow
\left \{
\begin {array} {l}
{ \dot{V}_u(x) \leq - \varphi(V(x))} \\
{ V(g(x,\xi))\leq  \alpha(V(x)),}
\end {array}
\right.
\end{equation}
for all $u \in \Uc$ with $u(0)=\xi$.
For a given input $u \in \Uc$, the Lie derivative is defined by
\begin{equation}
\label{LyapAbleitung_2}
\dot{V}_u(x)=\mathop{\overline{\lim}} \limits_{t \rightarrow +0} {\frac{1}{t}\big(V(\phi(t,x,u))-V(x)\big) }.
\end{equation}
%where $\phi_c$ is a transition map corresponding to the continuous part of the system $\Sigma$, i.e., $\phi_c(t,0,x,u)$ is a state of the system $\Sigma$ at time $t$, if the state at time $t_0:=0$ was $x$, input $u$ was applied and $T=\emptyset$.

If in addition
\begin{equation}
\label{Lyap_ExpFunk}
\varphi(s) = c s \mbox{ and } \alpha(s) = e^{-d} s
\end{equation}
for some $c,d \in \R$, then $V$ is called \emph{exponential ISS-Lyapunov function with rate coefficients} $c,d$.
\end{definition}

For any given impulse time sequence, the configuration of $\Sigma$ does not change for small enough times. Thus $\dot{V}_u(x)$ does not depend on the sequence of impulse times.

For a given sequence of impulse times, denote by $N(t,s)$ the number of jumps within the interval $(s,t]$.

Along the lines of \cite[Theorem 5]{DaM13b}, we have the following result
\begin{theorem}
\label{ExpCase}
Let $V$ be an exponential ISS-Lyapunov function for $\Sigma$ with corresponding coefficients $c \in \R$, $d \neq 0$. For arbitrary function $h:\R_+ \to(0,\infty)$, for which there exists $g \in \LL$: $h(x) \leq g(x)$ for all $x \in \R_+$
consider the class $\SSet[h]$ of impulse time-sequences, satisfying the generalized average dwell-time (gADT) condition:
\begin{equation}
\index{dwell-time condition!generalized average}
\index{generalized ADT}
\label{Dwell-Time-Cond}
-dN(t,s) - c(t-s) \leq  \ln h(t-s), \quad \forall t\geq s \geq t_0.
\end{equation}
%where $N(t,s)$ be a number of jumps during the time-span $[t,s)$.
Then the system $\Sigma$ is uniformly ISS over $\SSet[h]$.
\end{theorem}

%\mir{Special case with classical ADT.}

\subsection{Illustrative example}

Consider a cascade interconnection of a nonlinear scalar system
\begin{eqnarray}
\dot{z} = -z^3 + \|x\|_{X_Q}
\label{eq:Nonlinear-coupling}
\end{eqnarray}
with a live system consisting of an finite (but without a priori bound) number of identical exponentially stable subsystems of the form
\begin{align}
\label{eq:identical-agent}
\dot{x}_i	&= Ax_i,
\end{align}
where $i\in\N$, $A \in \R^{s\tm s}$ is a Hurwitz matrix, and $x_i \in \R^s$ for a certain $s \in \N$.

Thus, $S_0:=\N$, $X_i := \R^s$ for all $i\in S_0$, and we endow each $X_i$ with Euclidean norm. 
Now we set for each $Q \subset S_0$ 
\[
\|x\|_{X_Q}:= \Big(\sum_{j\in Q}|x_j|^2\Big)^{1/2},\quad x \in X_Q.
\]
Although we assume that the state of the system at each time belongs to a finite-dimensional space, the dimension of the state may grow to infinity. Thus, the choice of the norm on each $X_Q$ is crucial, even though the norms in finite-dimensional spaces are all equivalent.

Now consider an impulse time sequence $\T:=(t_k)_{k\in\N}$, and assume that at each time $t_k$ one new agent of the form \eqref{eq:identical-agent} is entering the system.
The initial state of the newly added agent at time $t_k$ is $u(t_k)$.

We are going to study the ISS of the system \eqref{eq:identical-agent}-\eqref{eq:Nonlinear-coupling}.

We start by analyzing the system \eqref{eq:identical-agent}. 
 As $A$ is Hurwitz, there is a positive definite matrix $P=P^T \in\R^{s \tm s}$ such that the Lyapunov inequality holds:
\[
PA+A^TP \leq -I.
\]

For any state $x \in X$, let $Q=Q(x) \subset S_0$ be such that $x \in X_Q$. 
Define the ISS Lyapunov function candidate as  
\begin{align}
\label{eq:ISS-LF-simple-quadratic}
V(x)	&:= \sum_{j\in Q} x_j^TPx_j,\quad x \in  X.
\end{align}
Clearly, the following sandwich estimates hold:
\begin{align*}
%\label{eq:ISS-LF-simple-quadratic}
\lambda_{\min}(P)\|x\|^2_{ X} &= \lambda_{\min}(P)\|x\|^2_{X_Q} \\
														  &\leq V(x) \leq \lambda_{\max}(P)\|x\|^2_{X_Q} = \lambda_{\max}(P)\|x\|^2_{ X},
\end{align*}
where $\lambda_{\min}(P), \lambda_{\max}(P)>0$ are the minimal and the maximal eigenvalue of the positive definite matrix $P$ respectively.

Take any $Q \subset S_0$ and any $x \in X_Q$.
Computing the Lie derivative of $V$ along the trajectories, we see that 
\begin{align}
\label{eq:ISS-LF-estimate-continuous-dynamics}
\dot{V}(x)	&= \sum_{j\in Q} x_j^T(PA + A^TP)x_j \nonumber\\
&\le - \sum_{j\in Q} x_j^Tx_j = -\|x\|^2_{ X} \leq -\frac{1}{\lambda_{\max}(P)}V(x).
\end{align}

Consider the impulsive dynamics.
At each time $t_k$, a new subsystem enters a system $\Sigma$, with the initial state $u(t_k)$, and thus the new state is
\[
x(t_k):=(x(t_k^-),u(t_k)).
\]
Thus, the dynamics of a Lyapunov function candidate $V$ at impulse times is
\begin{align*}
V((x,u))	&= V(x) + u^TPu \leq V(x) + \lambda_{\max}(P) |u|^2.
\end{align*}
This implies that 
\begin{eqnarray}
\label{eq:ISS-LF-estimate-impulsive-dynamics}
\lambda_{\max}(P) |u|^2 \leq \varepsilon V(x) \qrq  V((x,u)) \leq (1+\varepsilon)V(x).
\end{eqnarray}

Taking $c:= \lambda_{\max}(P)^{-1}$, $d:=-\ln(1+\varepsilon)$, and any $h$ as in the formulation of 
Theorem~\ref{ExpCase}, Theorem~\ref{ExpCase} ensures UISS of our system over the class of impulse time sequences $\SSet[h]$ satisfying 
\eqref{Dwell-Time-Cond}. 

The system \eqref{eq:Nonlinear-coupling} is ISS with $\|x\|_{X_Q}$ as an input, which can be seen by using the Lyapunov function $z \mapsto z^2$. Thus, \eqref{eq:identical-agent}-\eqref{eq:Nonlinear-coupling} is a cascade interconnection of ISS (uniformly over class of impulse time sequences $\SSet[h]$) systems, and thus it is ISS (uniformly over class of impulse time sequences $\SSet[h]$) using the same arguments as in the classical finite-dimensional case.

Note that there are impulse time sequences such that the system is not ISS with respect to these impulse time sequences. In fact, by adding new states to the system \eqref{eq:identical-agent} at an increasing rate, it is possible to let the state of the system \eqref{eq:identical-agent} converge to infinity as time goes to infinity.

\section{Discussion}

We have presented a formalization of control systems whose configuration may change in time, which 
we call live systems. 
%Due to the arrival and departure of the components, the system changes its configuration, which changes its dynamics. 
We understand the birth (or arrival) of new components as an input to the system and use the input-to-state stability framework to study stability of live systems.

If the map $q$ in \eqref{eq:Configuration-change} describing the change of configuration, depends on the state of a system, we need to go beyond the paradigm of impulsive systems, and consider live systems as hybrid systems (see \cite{GST12})  changing their configuration at impulse times. 

Having developed a foundation for live systems theory, we will have a rich framework to formulate and analyze new approaches for control. 

One of the applications of live systems theory is the development of \emph{plug-and-play control} methods that allow the controllers to work properly even though some parts of the system may detach from the system or become unfunctional or, conversely, some new parts/agents may join the network. In particular, one could consider the synchronization of open multi-agent systems, distributed control, and observation over networks with varying topology, etc. 
%A model-predictive plug-and-play control is another area related to P\ref{itm:P2.3}. The idea of MPC to recompute the controllers on a certain horizon rather than compute the optimal solution on the whole time interval fits naturally into the live systems point of view, where the changes in the configuration of the system may occur, which indicates that a recomputation of a controller is needed.

In plug-and-play control, a system can \emph{react to the changes} of the system configuration. 
A complementary problem is controlling a system by \emph{inducing the changes} of the system configuration. 
Examples are adding drones to the drones flock, deployment of new military units to the battlefield, adding predators to counteract parasites, etc.  
One can, however, go one step further and consider \emph{self-governing systems}, which redesign themselves by adding new elements, sensors, actuators, estimators, etc. 
As an example, consider optimal allocation models, which are an advanced class of mathematical models for the modeling of life histories of living organisms; see \cite{MiK14} and the references therein.
In a basic setting of these models, one considers a plant consisting of a fixed number of compartments and assumes that the plant can control itself to maximize fitness.
These models neglect the modularity, which makes them less exact and leads to false predictions \cite{Fox1992}.
By employing live systems theory, one could allow a plant to control not only the allocation of energy to existing compartments but also to control which modules have to be produced (brackets, flowers, leaves, etc.), at which time, and in which order.

\section*{References}

\bibliographystyle{abbrv}
%\bibliography{C:/Users/Andrii/Dropbox/TEX_Data/Mir_LitList_NoMir,C:/Users/Andrii/Dropbox/TEX_Data/MyPublications}
\bibliography{Mir_LitList_NoMir,MyPublications}

%\bibliography{C:/GoogleDrive/TEX_Data/Mir_LitList_NoMir,C:/GoogleDrive/TEX_Data/MyPublications}
% Important: no space in the list.

\end{document}